\documentclass[11pt,letterpaper,reqno]{amsart}

\usepackage[T1]{fontenc}
\usepackage{amsmath,amssymb,amsthm,mathtools}
\usepackage[colorlinks=true,linkcolor=blue,citecolor=blue,urlcolor=blue]{hyperref}
\newtheorem{theorem}{Theorem}[section]
\newtheorem{corollary}[theorem]{Corollary}
\newtheorem{lemma}[theorem]{Lemma}
\theoremstyle{remark}
\newtheorem{remark}[theorem]{Remark}

\newcommand{\C}{\mathbb C}
\newcommand{\D}{\mathcal D}

\newcommand{\dA}{\,dA}
\newcommand{\loc}{\mathrm{loc}}
\newcommand{\eps}{\varepsilon}

\DeclareMathOperator{\dist}{dist}

\begin{document}

\title[Infinitesimal Circular Morera Theorem]{An Infinitesimal Circular Morera Theorem}

\author[Q. Guo]{Qiteng Guo}
\address{College of Mathematics and Statistics, Northwest Normal University,
Lanzhou 730070, China}
\email{qiteng.guo@outlook.com}

\author[A. Xiao]{Ao Xiao}
\address{College of Mathematics and Statistics, Northwest Normal University,
Lanzhou 730070, China}
\email{xiaoao2math@gmail.com}


\begin{abstract}
We prove an infinitesimal circular version of Morera's theorem.  Let
\(D\subset\C\) be a domain and let \(f\in C(D)\).  If, at every
\(a\in D\),
\[
        \int_{|\zeta-a|=r}f(\zeta)\,d\zeta=o(r^2)\qquad (r\to0^+),
\]
then \(f\) is holomorphic in \(D\).  In particular, exact vanishing of all
sufficiently small centered circular integrals implies holomorphicity.  The
proof uses a local distributional \(\partial\)-primitive, a circular identity
for weak \(\partial\)-derivatives, and a pointwise asymptotic mean-value
criterion for harmonicity.
\end{abstract}

\maketitle

\section{Introduction}\label{sec:introduction}

Throughout the paper, a domain is a connected open subset of the complex
plane.  We use
\[
        \partial=\frac12(\partial_x-i\partial_y),\qquad
        \bar\partial=\frac12(\partial_x+i\partial_y).
\]
Morera's theorem is the classical converse to Cauchy's theorem: a continuous
complex-valued function on a domain is holomorphic once its integrals over a
sufficiently rich family of closed curves vanish.  In one standard formulation
it is enough to test boundaries of triangles contained in the domain; see, for
example, Ahlfors \cite{Ahlfors1979}.  Thus Morera's theorem converts an
integral condition, imposed on closed curves, into the differential conclusion
\(\bar\partial f=0\).

A natural refinement is to ask how much of the family of test curves is really
needed.  The problem considered here is a local circular version of this
question.  It was posed by D.~Gaier and L.~Zalcman and appeared as
Problem~7.28 in Anderson--Barth--Brannan
\cite[Problem~7.28]{AndersonBarthBrannan1977}; it is also recorded in
Hayman--Lingham's problem collection
\cite[Problem~7.28]{HaymanLingham2019}.  In that formulation one asks whether a
continuous function must be holomorphic if the integrals over all sufficiently
small circles centered at the point in question vanish, and also whether the
weaker infinitesimal condition of order \(o(r^2)\) is already sufficient.  The
same problem points to Zalcman's work on analyticity, the Pompeiu problem, and
mean-value equations \cite{Zalcman1972,Zalcman1973}.  For neighboring recent
work, Guo--He \cite{GuoHe2026} resolve the area-measure interpretation of
Hayman--Lingham Problem~7.29 by studying a variable-radius disk transform.
That problem concerns area integrals over center-dependent disks and is
geometrically distinct from the boundary-integral condition studied here, but
both belong to the same group of Zalcman problems on integral rigidity.  We use
this circle of results only as context; the proof below is local and does not
require the Fourier--Laplace machinery associated with the Pompeiu problem.

For \(a\in\C\) and \(r>0\), let
\[
        B(a,r)=\{z\in\C: |z-a|<r\},\qquad
        C(a,r)=\{z\in\C: |z-a|=r\},
\]
where \(C(a,r)\) is positively oriented.  If \(u\) is continuous in a
neighbourhood of \(\overline{B(a,r)}\), set
\begin{equation}\label{eq:def-mean-integral}
        A_u(a,r)=\frac1{2\pi}\int_0^{2\pi}u(a+re^{it})\,dt,
        \qquad
        J_u(a,r)=\int_{C(a,r)}u(\zeta)\,d\zeta .
\end{equation}
If \(D\subset\C\) is a domain, write
\(\rho_D(a)=\dist(a,\partial D)\), with the convention
\(\rho_D(a)=\infty\) when \(D=\C\).  All circular means and integrals below are
taken only for radii \(0<r<\rho_D(a)\).

With the notation in \eqref{eq:def-mean-integral}, the scale \(r^2\) is forced
by the differentiable case.  If \(f\in C^1(D)\),
Green's formula gives, whenever \(\overline{B(a,r)}\subset D\),
\begin{equation}\label{eq:intro-green}
        J_f(a,r)=\int_{C(a,r)} f(\zeta)\,d\zeta
        =2i\iint_{B(a,r)}\bar\partial f\,\dA .
\end{equation}
Consequently, \eqref{eq:intro-green} and the continuity of \(\bar\partial f\)
give
\begin{equation}\label{eq:intro-smooth-limit}
        \frac{J_f(a,r)}{r^2}\longrightarrow 2\pi i\,\bar\partial f(a)
        \qquad (r\to0^+).
\end{equation}
Thus the hypothesis \(J_f(a,r)=o(r^2)\) is exactly the infinitesimal assertion
\(\bar\partial f(a)=0\) when the classical derivative exists.  The point of the
main theorem is that no differentiability assumption is needed.

This leads to the following infinitesimal circular version of Morera's theorem.

\begin{theorem}\label{thm:main}
Let \(D\subset\C\) be a domain and let \(f\in C(D)\).  Suppose that, for every
\(a\in D\),
\begin{equation}\label{eq:main-hypothesis}
        J_f(a,r)=o(r^2)\qquad (r\to0^+,\ 0<r<\rho_D(a)).
\end{equation}
Then \(f\) is holomorphic in \(D\).
\end{theorem}

As an immediate consequence, we obtain the corresponding exact local circular
condition.

\begin{corollary}\label{cor:exact-circles}
Let \(D\subset\C\) be a domain and let \(f\in C(D)\).  Suppose that, for every
\(a\in D\), there is \(\delta(a)>0\) such that
\[
        J_f(a,r)=0
        \qquad\text{whenever }0<r<\min\{\delta(a),\rho_D(a)\}.
\]
Then \(f\) is holomorphic in \(D\).
\end{corollary}

\begin{proof}
The stated vanishing implies \eqref{eq:main-hypothesis} at every center.
\end{proof}

Theorem~\ref{thm:main} and Corollary~\ref{cor:exact-circles} give the asserted
affirmative answer to the two alternatives in the Gaier--Zalcman problem.

\subsection{Proof strategy}\label{subsec:proof-strategy}
The proof does not differentiate \(f\) directly.  On a relatively compact
subdomain we first construct a continuous function \(F\) satisfying
\(\partial F=f\) in the sense of distributions.  If \(F\) were smooth, a
polar-coordinate
calculation would show that the circular integral of \(\partial F\) is the
radial derivative of the circular mean of \(F\).  The same identity remains
valid under the weak hypothesis \(\partial F=f\):
\begin{equation}\label{eq:intro-circle-identity}
        A_F(a,R)-F(a)
        =-\frac{i}{\pi}\int_0^R\frac{J_f(a,s)}{s}\,ds .
\end{equation}
The assumption \(J_f(a,s)=o(s^2)\) then implies
\(A_F(a,R)-F(a)=o(R^2)\) as \(R\to0^+\) at every point.  A second-order
asymptotic form of the mean-value
characterization of harmonic functions shows that the real and imaginary parts
of \(F\) are harmonic.  Hence \(F\) is smooth, \(\partial F\) is holomorphic,
and the distributional identity \(\partial F=f\) becomes a pointwise identity by
continuity.

\subsection{Organization of the paper}\label{subsec:organization}
Section~\ref{sec:preliminaries} establishes the first two local tools used in
the argument---a distributional \(\partial\)-primitive and the weak circular
identity \eqref{eq:intro-circle-identity}---and recalls the asymptotic mean-value
criterion for harmonicity.  Section~\ref{sec:proof-main} proves
Theorem~\ref{thm:main}.
Section~\ref{sec:remarks} records two comments on the scale \(r^2\) and on the
local nature of the hypothesis.

\section{Preliminaries}\label{sec:preliminaries}

This section contains only local facts.  They are included in order to fix the
normalizations of \(\partial\), \(\bar\partial\), and \(d\zeta\), and to make the
passage from weak derivatives to circular means explicit.

We use the standard distributional fundamental-solution identities
\begin{equation}\label{eq:fundamental-solutions}
        \bar\partial\left(\frac1{\pi z}\right)=\delta_0,
        \qquad
        \partial\left(\frac1{\pi\bar z}\right)=\delta_0 .
\end{equation}
For background on distributions and fundamental solutions one may consult
Folland \cite[Chs.~0--1]{Folland1995} or H\"ormander \cite[Ch.~I]{Hormander1990}.
The first lemma is the corresponding local primitive construction for the
operator \(\partial\).

\begin{lemma}\label{lem:primitive}
Let \(\Omega\subset\C\) be open and let \(h\in C(\Omega)\).  If
\(V\Subset\Omega\), then there are an open set \(W\) and a function
\(U\in C(W)\) such that
\[
        \overline V\subset W\Subset\Omega,
        \qquad
        \partial U=h\quad\text{in }\D'(W).
\]
\end{lemma}

\begin{proof}
Choose an open set \(W\) with \(\overline V\subset W\Subset\Omega\), and choose
\(\chi\in C_c^\infty(\Omega)\) such that \(\chi=1\) near \(\overline W\).  Extend
\(\chi h\) by zero to \(\C\), put \(K(z)=1/(\pi\bar z)\), and set
\(U=K*(\chi h)\).  Since \(K\in L^1_\loc(\C)\) and \(\chi h\) is bounded and
compactly supported, continuity of translations in \(L^1_\loc(\C)\) gives
\(U\in C(\C)\).  By \eqref{eq:fundamental-solutions},
\[
        \partial U=(\partial K)*(\chi h)=\delta_0*(\chi h)=\chi h.
\]
Since \(\chi=1\) on \(W\), the identity \(\partial U=h\) holds in \(\D'(W)\).
\end{proof}

The next lemma is the analytic link between the circular integrals of a weak
\(\partial\)-derivative and the circular means of its primitive.  The proof is
written in two stages: first the smooth calculation, where the sign and constant
are visible, and then the standard mollification argument.

\begin{lemma}\label{lem:circle-identity}
Let \(\Omega\subset\C\) be open, let \(U,h\in C(\Omega)\), and assume that
\(\partial U=h\) in \(\D'(\Omega)\).
If \(a\in\Omega\), \(R>0\), and \(\overline{B(a,R)}\subset\Omega\), then
\begin{equation}\label{eq:circle-id}
        A_U(a,R)-U(a)
        =-\frac{i}{\pi}\int_0^R\frac{J_h(a,s)}{s}\,ds .
\end{equation}
The integral on the right is absolutely convergent near \(0\).
\end{lemma}

\begin{proof}
We first note the endpoint estimate.  Since \(h\) is continuous and
\(\int_{C(a,s)}d\zeta=0\), we have
\(J_h(a,s)=\int_{C(a,s)}(h(\zeta)-h(a))\,d\zeta\).  Put
\(\omega_h(a,s)=\sup_{|w-a|\le s}|h(w)-h(a)|\).  Then
\begin{equation}\label{eq:endpoint-bound}
        \left|\frac{J_h(a,s)}{s}\right|
        \le 2\pi\omega_h(a,s)\longrightarrow0
        \qquad (s\to0^+).
\end{equation}
Estimate \eqref{eq:endpoint-bound} shows that the right-hand side of
\eqref{eq:circle-id} is well defined at the lower endpoint.

Assume now that \(U\in C^1(\Omega)\) and \(h=\partial U\) classically.  With
\(\zeta=a+se^{it}\), polar coordinates about \(a\) give
\[
        \partial=\frac12e^{-it}\left(\partial_s-\frac{i}{s}\partial_t\right),
        \qquad
        d\zeta=ise^{it}\,dt .
\]
Therefore, for \(0<s<R\),
\begin{align*}
        J_h(a,s)
        &=\int_0^{2\pi}\partial U(a+se^{it})\,ise^{it}\,dt       \\
        &=\frac{is}{2}\int_0^{2\pi}\partial_sU(a+se^{it})\,dt
          +\frac12\int_0^{2\pi}\partial_tU(a+se^{it})\,dt          \\
        &=i\pi s\frac{d}{ds}A_U(a,s),
\end{align*}
where the \(\partial_t\)-term integrates to zero over one period.  Since
\(A_U(a,s)\to U(a)\) as \(s\to0^+\), integration from \(0\) to \(R\) yields
\eqref{eq:circle-id}.

It remains to pass to the stated weak case.  Choose \(R_1>R\) such that
\(\overline{B(a,R_1)}\subset\Omega\), and choose \(\chi\in C_c^\infty(\Omega)\)
such that \(\chi=1\) in a neighbourhood of \(\overline{B(a,R_1)}\).  Extend
\(\chi U\) and \(\chi h\) by zero to \(\C\).  Let \(\varphi_\eps\) be a standard
mollifier supported in \(B(0,\eps)\), and for \(0<\eps<R_1-R\) put
\[
        U_\eps=\varphi_\eps*(\chi U),
        \qquad
        h_\eps=\varphi_\eps*(\chi h).
\]
For points of \(B(a,R_1-\eps)\), the convolution sees only the region where
\(\chi=1\).  Consequently \(\partial U_\eps=h_\eps\) in
\(B(a,R_1-\eps)\).  Since
\(\overline{B(a,R)}\subset B(a,R_1-\eps)\), this identity holds in a
neighbourhood of \(\overline{B(a,R)}\), and the smooth identity applies to
\((U_\eps,h_\eps)\) at radius \(R\).

As \(\eps\to0\), the functions \(U_\eps\) and \(h_\eps\) converge uniformly to
\(U\) and \(h\), respectively, on \(\overline{B(a,R)}\).  Moreover, for
\(0<s\le R\),
\[
        \left|\frac{J_{h_\eps}(a,s)-J_h(a,s)}{s}\right|
        \le 2\pi\|h_\eps-h\|_{L^\infty(\overline{B(a,R)})}.
\]
Hence
\[
        \int_0^R\left|\frac{J_{h_\eps}(a,s)-J_h(a,s)}{s}\right|\,ds
        \le 2\pi R\|h_\eps-h\|_{L^\infty(\overline{B(a,R)})}\to0.
\]
The right-hand sides of the mollified identities therefore converge to the
right-hand side of \eqref{eq:circle-id}; the left-hand sides converge by the
uniform convergence of \(U_\eps\).  This proves the identity.
\end{proof}

The last preliminary result is the classical Blaschke asymptotic mean-value
criterion.  The planar form below is precisely Kuznetsov
\cite[Thm.~1.6]{Kuznetsov2019}.  For related mean-value and
generalized-Laplacian results, see Axler--Bourdon--Ramey
\cite[Thms.~1.4 and~1.24]{AxlerBourdonRamey2001}, Netuka--Vesel\'y
\cite{NetukaVesely1994}, Pokrovskii \cite{Pokrovskii2017}, and
C\'ordoba--Oc\'ariz \cite{CordobaOcariz2020}.

\begin{lemma}\label{lem:asymptotic-mean}
Let \(\Omega\subset\C\) be open and let \(u\in C(\Omega)\) be real-valued.  If,
for every \(a\in\Omega\),
\begin{equation}\label{eq:asymptotic-mvp}
        A_u(a,r)-u(a)=o(r^2)\qquad (r\to0^+),
\end{equation}
then \(u\) is harmonic in \(\Omega\).
\end{lemma}

\section{Proof of the main theorem}\label{sec:proof-main}

Let \(V\Subset D\) be an arbitrary open set.  By Lemma~\ref{lem:primitive},
applied with \(h=f\), there are an open set \(W\) with
\(\overline V\subset W\Subset D\) and a function \(F\in C(W)\) such that
\begin{equation}\label{eq:partial-F-f}
        \partial F=f\quad\text{in }\D'(W).
\end{equation}

Fix \(a\in V\).  For all sufficiently small \(r>0\),
\(\overline{B(a,r)}\subset V\).  Lemma~\ref{lem:circle-identity}, used with
\(U=F\) and \(h=f\), gives
\begin{equation}\label{eq:mean-F}
        A_F(a,r)-F(a)
        =-\frac{i}{\pi}\int_0^r\frac{J_f(a,s)}{s}\,ds .
\end{equation}
The hypothesis \eqref{eq:main-hypothesis} is pointwise in the center.  Hence,
for this fixed \(a\) and for every \(\eta>0\), there is \(\delta>0\) such that
\(|J_f(a,s)|\le \eta s^2\) whenever \(0<s<\delta\).  If \(0<r<\delta\), then
\eqref{eq:mean-F} yields
\[
        |A_F(a,r)-F(a)|
        \le \frac1\pi\int_0^r \eta s\,ds
        =\frac{\eta}{2\pi}r^2 .
\]
Since \(\eta\) is arbitrary,
\begin{equation}\label{eq:F-asymptotic-mean}
        A_F(a,r)-F(a)=o(r^2)\qquad (r\to0^+)
\end{equation}
at the point \(a\).  As \(a\in V\) was arbitrary, \eqref{eq:F-asymptotic-mean}
holds at every point of \(V\).

By \eqref{eq:F-asymptotic-mean}, the hypothesis
\eqref{eq:asymptotic-mvp} of Lemma~\ref{lem:asymptotic-mean} holds for both
\(\operatorname{Re}F\) and \(\operatorname{Im}F\).  Hence \(F\) is harmonic
componentwise in \(V\).  Therefore \(F\in C^\infty(V)\), and
\begin{equation}\label{eq:dbar-partial-F-zero}
        \bar\partial(\partial F)=\frac14\Delta F=0
\end{equation}
classically in \(V\).  By \eqref{eq:dbar-partial-F-zero}, \(\partial F\) is
holomorphic in \(V\).

It remains only to identify \(\partial F\) with the original continuous
function \(f\).  By \eqref{eq:partial-F-f}, the two functions agree as
distributions on \(V\).  Hence they agree almost everywhere.  Since \(f\) is
continuous and \(\partial F\) is smooth, the difference \(f-\partial F\) is a
continuous function that vanishes almost everywhere; it must vanish everywhere.
Thus \(f=\partial F\) pointwise on \(V\), and \(f\) is holomorphic on \(V\).
Because \(V\Subset D\) was arbitrary, \(f\) is holomorphic in \(D\).  This proves
Theorem~\ref{thm:main}.

\section{Final remarks}\label{sec:remarks}

We first make explicit the coefficient detected by the normalized circular
integral in the differentiable case.

\begin{remark}\label{rem:coefficient}
Formula \eqref{eq:intro-smooth-limit} shows that, for \(C^1\) functions, the
quantity \(r^{-2}J_f(a,r)\) detects the value of \(2\pi i\,\bar\partial f(a)\).
For example, if \(f(z)=\bar z\), then
\(\int_{|\zeta-a|=r}\bar\zeta\,d\zeta=2\pi i\,r^2\).  Thus the order \(r^2\)
in Theorem~\ref{thm:main} is the first nontrivial order
at which non-holomorphicity can be seen in the smooth case.
\end{remark}

We close by emphasizing the entirely local nature of the argument.

\begin{remark}\label{rem:locality}
The proof uses only the little-oh condition at each fixed center \(a\).  No
uniformity in \(a\) is assumed.  The construction of the \(\partial\)-primitive
is also local: it is made on relatively compact subdomains, and no global
topological condition on \(D\) is involved.
\end{remark}

\end{document}